\documentclass{article}%
\usepackage{amsfonts}
\usepackage{amsmath}
\usepackage{amssymb}
\usepackage{graphicx}%
\providecommand{\U}[1]{\protect\rule{.1in}{.1in}}
\newtheorem{theorem}{Theorem}

\newtheorem{conclusion}[theorem]{Conclusion}

\newtheorem{definition}[theorem]{Definition}

\newtheorem{proposition}[theorem]{Proposition}

\newenvironment{proof}[1][Proof]{\noindent\textbf{#1.} }{\ \rule{0.5em}{0.5em}}
\begin{document}

\title{The Segal--Bargmann transform for unitary groups in the large-$N$ limit}
\author{Brian C. Hall\thanks{Supported in part by National Science Foundation grants
DMS-1001328 and DMS-1301534.}\\University of Notre Dame}
\maketitle

\begin{abstract}
This paper describes results of the author with B. K. Driver and T. Kemp
concerning the large-$N$ limit of the Segal--Bargmann transform for the
unitary group $U(N).$ We consider the transform on matrix-valued functions
that are polynomials in a single variable in $U(N).$ We show that in the
large-$N$ limit, the transform maps functions of this type to single-variable
polynomial functions on the complex group $GL(N;\mathbb{C}).$ This result was
conjectured by Ph. Biane and was also proved independently by G. C\'{e}bron.

The first main ingredient in our proof of this result is an \textquotedblleft
asymptotic product rule\textquotedblright\ for the Laplacian on $U(N),$ which
allows us to compute explicitly the leading-order large-$N$ behavior of the
heat operator on $U(N).$ The second main ingredient in the proof is the
phenomenon of \textquotedblleft concentration of traces,\textquotedblright\ in
which the relevant heat kernel measures are concentrating onto sets where the
trace of any power of the variable is constant.

\end{abstract}

\section{Overview}

Let $U(N)$ denote the group of $N\times N$ unitary matrices and let
$GL(N;\mathbb{C})$ denote the group of all invertible $N\times N$ matrices
with complex entries. Then $GL(N;\mathbb{C})$ is the \textquotedblleft
complexification\textquotedblright\ of $U(N)$ in the sense of \cite[Sect.
3]{H1}. Let $\rho_{t}^{N}$ and $\mu_{t}^{N}$ denote the heat kernel measures
on $U(N)$ and $GL(N;\mathbb{C}),$ respectively, based at the identity and
defined with respect to certain left-invariant Riemannian metrics. Let
$\mathcal{H}L^{2}(GL(N;\mathbb{C}),\mu_{t}^{N})$ denote the space of
holomorphic functions on $GL(N;\mathbb{C})$ that are square integrable with
respect to $\mu_{t}^{N}.$ The \textbf{Segal--Bargmann transform} $B_{t}^{N}$
for $U(N),$ as introduced in \cite{H1}, is a unitary map of $L^{2}%
(U(N),\rho_{t}^{N})$ onto $\mathcal{H}L^{2}(GL(N;\mathbb{C}),\mu_{t}^{N}).$
The transform itself is defined by
\[
B_{t}^{N}(f)=(e^{t\Delta/2}f)_{\mathbb{C}},
\]
where $e^{t\Delta/2}$ is the time-$t$ forward heat operator and $(\cdot
)_{\mathbb{C}}$ denotes analytic continuation from $U(N)$ to $GL(N;\mathbb{C}%
).$ (See \cite{Bull} for a survey of this and related constructions.)

We may extend the transform to act on functions on $U(N)$ with values in
$M_{N}(\mathbb{C}),$ space of all $N\times N$ matrices with complex entries.
The extension is accomplished by applying the scalar transform
\textquotedblleft entrywise.\textquotedblright\ We denote the resulting
\textbf{boosted\ Segal--Bargmann transform} by $\mathbf{B}_{t}^{N}$; it maps
$L^{2}(U(N),\rho_{t}^{N};M_{N}(\mathbb{C}))$ onto $\mathcal{H}L^{2}%
(GL(N;\mathbb{C}),\mu_{t}^{N};M_{N}(\mathbb{C})).$ As proposed by Philippe
Biane in \cite{Biane2}, we apply $\mathbf{B}_{t}^{N}$ to
\textbf{single-variable polynomial functions} on $U(N)$\ that is, functions of
the form%
\begin{equation}
f(U)=c_{0}I+c_{1}U+c_{2}U^{2}+\cdots+c_{N}U^{N},\quad U\in U(N), \label{polyU}%
\end{equation}
where $c_{0},\ldots,c_{N}$ are constants.

If we apply $\mathbf{B}_{t}^{N}$ to such a polynomial function, the result
will typically \textit{not} be a polynomial function on $GL(N;\mathbb{C}).$
Rather, the result will be a \textbf{trace polynomial function}\ on
$GL(N;\mathbb{C}),$ that is, a linear combination of functions of the form%
\begin{equation}
Z^{k}\mathrm{tr}(Z)\mathrm{tr}(Z^{2})\cdots\mathrm{tr}(Z^{M}),\quad Z\in
GL(N;\mathbb{C}), \label{tracePoly}%
\end{equation}
where $k$ and $M$ are non-negative integers. Here $\mathrm{tr}(\cdot)$ is the
\textbf{normalized trace} given by%
\begin{equation}
\mathrm{tr}(A)=\frac{1}{N}\sum_{j=1}^{N}A_{jj} \label{normalizedTrace}%
\end{equation}
for any $A\in M_{N}(\mathbb{C}).$

Although for any one fixed value of $N,$ the boosted transform $\mathbf{B}%
_{t}^{N}$ does not map polynomial functions on $U(N)$ to polynomial functions
on $GL(N;\mathbb{C}),$ there is a sense in which \textbf{the large-}%
$N$\textbf{\ limit} of $\mathbf{B}_{t}^{N}$ does have this property. To
understand how this works, let consider the example of the matrix-valued
function
\[
f(U)=U^{2}
\]
on $U(N).$ Then, according to Example 3.5 of \cite{DHK}, we have%
\begin{equation}
\mathbf{B}_{t}^{N}(f)(Z)=e^{-t}\left[  \cosh(t/N)Z^{2}-t\frac{\sinh(t/N)}%
{t/N}Z\mathrm{tr}(Z)\right]  . \label{btUsquared1}%
\end{equation}

If we formally let $N$ tend to infinity in (\ref{btUsquared1}), we obtain%
\begin{equation}
\lim_{N\rightarrow\infty}\mathbf{B}_{t}^{N}(f)(Z)=e^{-t}[Z^{2}-tZ\mathrm{tr}%
(Z)]. \label{btUsquared2}%
\end{equation}
The right-hand side of (\ref{btUsquared2}) is, apparently, still a trace
polynomial and not a single-variable polynomial as in (\ref{polyU}). There is,
however, another limiting phenomenon that occurs when $N$ tends to infinity,
in addition to the convergence of the coefficients of $Z^{2}$ and
$Z\mathrm{tr}(Z)$ in (\ref{btUsquared1}), namely, the phenomenon of
\textbf{concentration of trace}.

As $N$ tends to infinity, the function $\mathrm{tr}(U^{k})$ in $L^{2}%
(U(N),\rho_{t}^{N})$ converges as $N$ tends to infinity to a certain constant
$\nu_{k}(t),$ in the sense that%
\[
\lim_{N\rightarrow\infty}\left\Vert \mathrm{tr}(U^{k})-\nu_{k}(t)\right\Vert
_{L^{2}(U(N),\rho_{t}^{N})}=0.
\]
What this means, more accurately, is that the \textit{measure} $\rho_{t}^{N} $
on $U(N)$ is concentrating, as $N$ tends to infinity with $t$ fixed, onto the
set where~$\mathrm{tr}(U^{k})=\nu_{k}(t).$ A similar concentration of trace
phenomenon occurs in the space $\mathcal{H}L^{2}(GL(N;\mathbb{C}),\mu_{t}%
^{N}),$ except that in this case, all of the traces concentrate to the value
1:%
\[
\lim_{N\rightarrow\infty}\left\Vert \mathrm{tr}(Z^{k})-1\right\Vert
_{L^{2}(GL(N;\mathbb{C}),\mu_{t}^{N})}=0.
\]
(See Theorem \ref{concentrate.thm} for a more general version of these
concentration results.)

Thus, the \textquotedblleft correct\textquotedblright\ way to evaluate the
large-$N$ limit in (\ref{btUsquared1}) is in two stages. First, we take the
limit as $N$ tends to infinity of the coefficients of $Z^{2}$ and
$Z\mathrm{tr}(Z),$ as in (\ref{btUsquared2}). Second, we replace
$\mathrm{tr}(Z)$ by the constant 1. The result is%
\begin{equation}
\lim_{N\rightarrow\infty}\mathbf{B}_{t}^{N}(f)(Z)=e^{-t}[Z^{2}-tZ].
\label{btUsquared3}%
\end{equation}
Note that the right-hand side of (\ref{btUsquared3}) is, for each fixed value
of $t,$ a polynomial in $Z.$

In \cite{DHK}, we show that a similar phenomenon occurs in general. Given any
polynomial $p$ in a single variable, let $p_{N}$ denote the matrix-valued
function on $U(N)$ obtained by plugging a variable $U\in U(N)$ into $p,$ as in
(\ref{polyU}). We also allow $p_{N}$ to denote the similarly defined function
on $GL(N;\mathbb{C}).$

\begin{theorem}
[Driver--Hall--Kemp]\label{mainIntro.thm}Let $p$ be a polynomial in a single
variable. Then for each fixed $t>0,$ there exists a unique polynomial $q_{t}$
in a single variable such that%
\begin{equation}
\lim_{N\rightarrow\infty}\left\Vert \mathbf{B}_{t}^{N}(p_{N})-(q_{t}%
)_{N}\right\Vert _{L^{2}(GL(N;\mathbb{C}),\mu_{t}^{N};M_{N}(\mathbb{C}))}=0.
\label{introLim}%
\end{equation}

\end{theorem}

If, for example, $p$ is the polynomial $p(u)=u^{2},$ then $q_{t}$ is the
polynomial given by%
\[
q_{t}(z)=e^{-t}(z^{2}-tz),
\]
so that
\[
(q_{t})_{N}(Z)=e^{-t}(Z^{2}-tZ),\quad Z\in GL(N;\mathbb{C}),
\]
as on the right-hand side of (\ref{btUsquared3}).

In \cite{DHK}, we also show that the map $p\mapsto q_{t}$ coincides with the
\textquotedblleft free Hall transform\textquotedblright\ of Biane, denoted
$\mathcal{G}^{t}$ in \cite{Biane2}. Although it was conjectured in
\cite{Biane2} that $\mathcal{G}^{t}$ is the large-$N$ limit of $\mathbf{B}%
_{t}^{N} $ as in (\ref{introLim}), Biane actually constructs $\mathcal{G}^{t}
$ by using a free probability version of the analysis of Gross--Malliavin
\cite{GrossMalliavin}. Theorem \ref{mainIntro.thm} was also proved
independently by G. Cebr\'{o}n \cite{Ceb}, using substantially different
methods. Besides using very different methods from \cite{Ceb}, the paper
\cite{DHK} establishes a \textquotedblleft two-parameter\textquotedblright%
\ version of Theorem \ref{mainIntro.thm}, as described in Section
\ref{twoParam.sec}. Section \ref{induct.sec} gives an inductive procedure for
computing the polynomials $q_{t}$ in Theorem \ref{mainIntro.thm}.

A key tool in proving the results described above is the \textbf{asymptotic
product rule} for the Laplacian on $U(N).$ This rule states that---on certain
classes of functions and for large values of $N$---the Laplacian behaves like
a \textit{first-order} differential operator. That is to say, in the usual
product rule for the Laplacian, the cross terms are small compared to the
other two terms. The asymptotic product rule provides the explanation for the
concentration of trace phenomenon and is also the key tool we use in deriving
a recursive formula for the polynomials $q_{t}$ in Theorem \ref{mainIntro.thm}.

\section{The Laplacian and Segal--Bargmann transform on $U(N)$}

We consider $U(N),$ the group of $N\times N$ unitary matrices. The Lie algebra
$u(N)$ of $U(N)$ is the $N^{2}$-dimensional real vector space consisting of
$N\times N$ matrices $X$ with $X^{\ast}=-X.$ We use on $u(N)$ the following
(real) inner product $\left\langle \cdot,\cdot\right\rangle _{N}$:%
\begin{equation}
\left\langle X,Y\right\rangle _{N}=N\operatorname{Re}(\mathrm{Trace}(X^{\ast
}Y)), \label{unInner}%
\end{equation}
where $\mathrm{Trace}$ is the ordinary trace, $\mathrm{Trace}(A)=\sum
_{j}A_{jj}.$ The motivation for the scaling by a factor of $N$ will be
explained shortly.

We think of the Lie algebra $u(N)$ as being the tangent space at the identity
of the manifold $U(N).$ We can then extend the inner product (\ref{unInner})
on $u(N)$ uniquely to a left-invariant Riemannian structure on $U(N).$
Actually, since the inner product in (\ref{unInner}) is invariant under the
adjoint action of $U(N),$ this Riemannian structure is bi-invariant, that is,
invariant under both the left and right actions of $U(N)$ on itself.

Associated to the Riemannian structure on $U(N)$ there is the Laplacian
$\Delta_{N}$, which we take to be a \textit{negative} operator. Let me
emphasize that $\Delta_{N}$ is always defined with respect to the Riemannian
structure whose value at the identity is given by (\ref{unInner}), with the
scaling by a factor of $N.$ For any $X\in u(N),$ we can define a
left-invariant vector field $\tilde{X}$ on $U(N)$ by the formula%
\begin{equation}
(\tilde{X}f)(U)=\left.  \frac{d}{dt}f\left(  Ue^{tX}\right)  \right\vert
_{t=0}. \label{leftInv}%
\end{equation}
If $\{X_{j}\}$ is an orthonormal basis for $u(N)$ with respect to the inner
product (\ref{unInner}), then $\Delta_{N}$ may be computed as%
\[
\Delta_{N}=\sum_{j=1}^{N^{2}}\tilde{X}_{j}^{2}.
\]

As a simple example, we may consider the action of $\Delta_{N}$ on the matrix
entries for the standard representation of $U(N),$ that is, functions of the
form $f_{jk}(U)=U_{jk}.$ It follows from the $k=1$ case of Proposition
\ref{lapPower.prop} below that%
\begin{equation}
\Delta_{N}(U_{jk})=-U_{jk}. \label{eigen}%
\end{equation}
That is, the functions $f_{jk}$ are eigenvalues for $\Delta_{N}$ with
eigenvalue $-1,$ for all $N$ and all $j,k.$ In particular, the normalization
of the inner product in (\ref{unInner}) has the result that the eigenvalues of
$\Delta_{N}$ in the standard representation are \textit{independent of }$N.$
By contrast, if we had omitted the factor of $N$ in (\ref{unInner}), we would
have had $\Delta_{N}(U_{jk})=-NU_{jk},$ which would not bode well for trying
to take the $N\rightarrow\infty$ limit. Note that the inner product and the
Laplacian scale oppositely; the factor of $N$ in (\ref{unInner}) produces a
factor of $1/N$ in the formula for $\Delta_{N},$ which scales the eigenvalues
from $-N$ to $-1.$

For any $t>0,$ let $e^{t\Delta_{N}/2}$ denote the time-$t$ (forward) heat
operator. If $P_{t}^{N}$ denotes the \textbf{heat kernel} at the identity on
$U(N),$ then we may compute the heat operator as%
\[
(e^{t\Delta_{N}/2}f)(U)=\int_{U(N)}P_{t}^{N}(UV^{-1})f(V)~dV,
\]
where $dV$ is the Riemannian volume measure on $U(N),$ which is a bi-invariant
Haar measure. It is shown in Section 4 of \cite{H1} that for each fixed $t>0,$
the function $P_{t}^{N}$ admits a unique holomorphic extension from $U(N)$ to
the complex group $GL(N;\mathbb{C}).$ Here, $GL(N;\mathbb{C})$ is the
complexification of $U(N)$ in the sense of Section 3 of \cite{H1}.

The paper \cite{H1} considers, more generally, any connected compact Lie group
with a bi-invariant Riemannian metric. In particular, we could replace the
inner product (\ref{unInner}) on $u(N)$ by any other multiple of the real
Hilbert--Schmidt inner product. The particular scaling of the inner product in
(\ref{unInner}) turns out, however, to be the right one for taking the
large-$N$ limit. We will explore this matter further is Section \ref{lim.sec}.

We consider also the \textbf{heat kernel measure }$\rho_{t}^{N}$ (based at the
identity) on $U(N),$ given by%
\[
d\rho_{t}^{N}(U)=P_{t}^{N}(U)~dU,
\]
where $dU$ is the Riemannian volume measure on $U(N),$ and the associated
Hilbert space, $L^{2}(U(N),\rho_{t}^{N}).$ Let $\mathcal{H}(GL(N;\mathbb{C}))$
denote the space of holomorphic functions on $GL(N;\mathbb{C}).$ For each
fixed $t>0,$ the \textbf{Segal--Bargmann transform} is then linear map
\[
B_{t}^{N}:L^{2}(U(N),\rho_{t}^{N})\rightarrow\mathcal{H}(GL(N;\mathbb{C}))
\]
given by%
\[
(B_{t}^{N}f)(Z)=\int_{U(N)}P_{t}^{N}(ZV^{-1})f(V)~dV,\quad Z\in
GL(N;\mathbb{C}),
\]
where $P_{t}^{N}(ZV^{-1})$ refers to the holomorphic extension of $P_{t}^{N}$
from $U(N)$ to $GL(N;\mathbb{C}).$ Equivalently, we may write $B_{t}^{N}f$ as%
\[
B_{t}^{N}f=(e^{t\Delta_{N}/2}f)_{\mathbb{C}}
\]
where $(\cdot)_{\mathbb{C}}$ denotes the analytic continuation of a function
from $U(N)$ to $GL(N;\mathbb{C}).$

The expression (\ref{unInner}) also defines a real-valued inner product on the
Lie algebra $gl(N;\mathbb{C})$ of $GL(N;\mathbb{C}).$ This inner product then
determines a left-invariant Riemannian metric on $GL(N;\mathbb{C}).$ We let
$\mu_{t}^{N}$ denote the associated heat kernel measure on $GL(N;\mathbb{C}),$
based at the identity.

\begin{theorem}
For each $t>0,$ the map $B_{t}^{N}$ is a unitary map of $L^{2}(U(N),\rho
_{t}^{N})$ onto $\mathcal{H}L^{2}(GL(N;\mathbb{C}),\mu_{t}^{N}),$ where
$\mathcal{H}L^{2}$ denotes the space of square-integrable holomorphic functions.
\end{theorem}

This result is Theorem $1^{\prime}$ in \cite{H1}. The result holds more
generally for an arbitrary compact Lie group $K$ together with its
complexification $K_{\mathbb{C}}.$ If one performs the analogous construction
on the commutative Lie group $\mathbb{R}^{n},$ one obtains (modulo minor
differences of normalization) the classical Segal--Bargmann considered by
Segal \cite{Segal1,Segal2} and Bargmann \cite{Bargmann}. Actually, one can
construct a unitary Segal--Bargmann transform for connected Lie groups of
\textquotedblleft compact type,\textquotedblright\ a class that includes both
$\mathbb{R}^{n}$ and $U(N).$ (See \cite{Driver} and \cite{newform}.)

We may extend the transform to a \textquotedblleft boosted\textquotedblright%
\ transform $\mathbf{B}_{t}^{N},$ acting on functions $f:U(N)\rightarrow
M_{N}(\mathbb{C}),$ by applying the scalar transform $B_{t}^{N}$
\textquotedblleft componentwise.\textquotedblright\ That is, $\mathbf{B}%
_{t}^{N}f$ is the holomorphic function $F:GL(N;\mathbb{C})\rightarrow
M_{N}(\mathbb{C})$ whose $(j,k)$ entry is $B_{t}^{N}(f_{jk}).$ We define the
norm of matrix-valued functions on $U(N)$ or $GL(N;\mathbb{C})$ as follows:
\begin{align}
\left\Vert f\right\Vert _{L^{2}(U(N),\rho_{t}^{N};M_{N}(\mathbb{C}))}^{2}  &
=\int_{U(N)}\mathrm{tr}(f(U)^{\ast}f(U))~d\rho_{t}^{N}(U)\label{MnNorm}\\
\left\Vert f\right\Vert _{L^{2}(GL(N;\mathbb{C}),\mu_{t}^{N};M_{N}%
(\mathbb{C}))}^{2}  &  =\int_{GL(N;\mathbb{C})}\mathrm{tr}(f(Z)^{\ast
}f(Z))~d\mu_{t}^{N}(Z), \label{MnNorm2}%
\end{align}
where $\mathrm{tr}(\cdot)$ is the normalized trace defined in
(\ref{normalizedTrace}). Note that the normalization of the Hilbert--Schmidt
norm in (\ref{MnNorm}) and (\ref{MnNorm2}) is different from the one we use in
(\ref{unInner}) to define the Laplacian $\Delta_{N}.$ The normalizations in
(\ref{MnNorm}) and (\ref{MnNorm2}) ensure that the norm of the constant
function $f(U)=I$ is 1 in either Hilbert space.

\section{The large-$N$ limit\label{lim.sec}}

Since Segal's work on the Segal--Bargmann transform was for an
infinite-dimensional Euclidean space, it is natural to try for a version of
the transform for infinite-dimensional Lie groups. In \cite{HallSengupta},
Sengupta and I construct one example of such a transform, for the path group
over a compact Lie group. The results in \cite{HallSengupta} are motivated by
earlier results of Gross \cite{GrossUniqueness} and Gross--Malliavin
\cite{GrossMalliavin}.

One could also attempt to take the limit of the Segal--Bargmann transform on a
nested family of compact Lie groups, such as $U(N).$ Indeed, the study of the
large-$N$ limit of the heat kernel on $U(N)$ also arises in the study of the
\textquotedblleft master field\textquotedblright\ on the plane, which is the
large-$N$ limit of (Euclidean) Yang--Mills theory with structure group $U(N).$
(See \cite{Singer}, \cite{Sengupta}, \cite{AnshelevichSengupta}, \cite{levy}
for mathematical results concerning the master field in the plane.) Although
the Segal--Bargmann transform is not typically part of this analysis (but see
\cite{AHS}), some of the same methods that we use in \cite{DHK} are employed
in the study of the master field.

The most obvious approach to the large-$N$ limit for the Segal--Bargmann
transform on $U(N)$ would be to use the real Hilbert--Schmidt inner product on
each Lie algebra $u(N)$:
\begin{equation}
\left\langle X,Y\right\rangle =\operatorname{Re}[\mathrm{Trace}(X^{\ast}Y)],
\label{unnormalized}%
\end{equation}
This approach is natural in that the inner product on $u(N)$ agrees with the
restriction to $u(N)\subset u(N+1)$ of the inner product on $u(N+1).$

Results of M. Gordina \cite{Go1,Go2}, however, show that this approach does
not work. Let $\gamma_{t}^{N}$ denote the heat kernel measure on
$GL(N;\mathbb{C})$ with respect to the metric determined by
(\ref{unnormalized}). (We reserve the notation $\mu_{t}^{N}$ for the heat
kernel with respect to the metric determined by (\ref{unInner}).) Gordina's
approach is to study the target space for the Segal--Bargmann transform,
$\mathcal{H}L^{2}(GL(N;\mathbb{C}),\gamma_{t}^{N}).$ Let us assume, for the
moment, that the measures $\gamma_{t}^{N}$ on $GL(N;\mathbb{C})$ have a
reasonable large-$N$ limit $\gamma_{t}^{\infty}$ on some \textquotedblleft
version\textquotedblright\ of $GL(\infty;\mathbb{C}).$ (We might interpret
$GL(\infty;\mathbb{C})$ as being, for example, the group of all bounded,
invertible operators on a Hilbert space.) One would then expect to be able to
compute the norm of elements of $\mathcal{H}L^{2}(GL(\infty;\mathbb{C}%
),\gamma_{t}^{\infty})$ by the Taylor expansion method of Driver and Gross
\cite{Driver,DriverGross}. This method expresses the $L^{2}$ norm of a
holomorphic function $F$ on a complex Lie group, with respect to a heat kernel
measure, as a certain sum of squares of left-invariant derivatives of $F$,
evaluated at the identity. Gordina shows that if one uses the Hilbert--Schmidt
norm on the Lie algebra $gl(\infty;\mathbb{C}),$ then the relevant sum of
squares of derivatives is always infinite, unless the holomorphic function in
question is constant. (See Theorem 8.1 in \cite{Go1}.)

We see, then, that there cannot be any nonconstant holomorphic functions on
$GL(\infty;\mathbb{C})$ that have finite $L^{2}$ norm with respect to the
hypothetical limiting measure $\gamma_{t}^{\infty}.$ This result is presumably
telling us that there is, in fact, no limiting measure $\gamma_{t}^{\infty}$
in the first place. Thus, the target space of the hoped-for Segal--Bargmann
transform for $U(\infty)$ is not well defined.

The preceding discussion shows that if we use the \textit{un-normalized}
Hilbert--Schmidt inner product (\ref{unnormalized}) on $u(N)$---and thus also
on $gl(N;\mathbb{C})$---then we do not obtain a well-defined Segal--Bargmann
transform in the $N\rightarrow\infty$ limit. This fact motivates the
introduction of the \textit{normalized} Hilbert--Schmidt inner product
(\ref{unInner}) that we will use throughout the remainder of the paper. Recall
that with the normalization of the inner product in (\ref{unInner}), we have
$\Delta_{N}(U_{jk})=-U_{jk}.$ That is, the factor of $N$ in (\ref{unInner})
(which translates into a factor of $1/N$ in the associated Laplacian) keeps
the eigenvalues of $\Delta_{N}$ in the matrix entries from blowing up as $N$
tends to infinity, which gives us some hope of obtaining a well-defined
transform in the limit.

\section{Concentration properties of the heat kernel measures}

In \cite{Biane2}, Biane proposed studying the large-$N$ behavior of the
Segal--Bargmann transform on $U(N)$ using the normalization of the inner
product in (\ref{unInner}). Biane also introduced in \cite{Biane2} the idea of
studying the transform on a certain very special class of matrix-valued
functions on $U(N),$ namely the single-variable polynomial functions in
(\ref{polyU}). A main result of \cite{DHK}, which was conjectured in
\cite{Biane2}, is that \textit{in the large-}$N$\textit{\ limit}, such
functions map to single-variable polynomial functions on $GL(N;\mathbb{C}).$
(See Theorem \ref{mainIntro.thm} in the overview.) In this section, we try to
understand this result from a conceptual standpoint, by looking into the
large-$N$ behavior of the heat kernel measures $\rho_{t}^{N}$ on $U(N)$ and
$\mu_{t}^{N}$ on $GL(N;\mathbb{C}).$ With Biane's scaling of the metrics,
these measures have interesting concentration properties for large $N,$ which
help explain the large-$N$ behavior of the Segal--Bargmann transform. Although
the results of this section are not actually used in the proof of Theorem
\ref{mainIntro.thm}, they provide a helpful way of thinking about \textit{why}
that theorem should hold.

Recall that rescaling the inner product on $u(N)$ by a factor of $N$ (as in
(\ref{unInner})) has the effect of rescaling the Laplacian by a factor of
$1/N.$ This rescaling is designed to keep the Laplacian and heat operator from
blowing up as $N$ tends to infinity. In some sense, however, the rescaling
does \textit{too} good a job of controlling things, in that the limiting
transform is well defined but, on certain classes of functions, trivial.
Biane's passage to matrix-valued functions allows the large-$N$ limit to be
both well defined and interesting.

Let us now look more closely into these issues. Results of Biane
\cite{Biane1}, E. Rains \cite{Rains}, and T. Kemp \cite{Kemp} may be
interpreted as saying that, in the large-$N$ limit, the heat kernel measure
$\rho_{t}^{N}$ on $U(N)$ \textit{concentrates onto a single conjugacy class}.
To make this claim more precise, let us note that the conjugacy class of a
matrix $U\in U(N)$ is determined by the list $\lambda_{1},\ldots,\lambda_{n}$
of its eigenvalues, where $\left\vert \lambda_{j}\right\vert =1.$ This list of
eigenvalues can be encoded into the \textit{empirical eigenvalue distribution}
of $U,$ which is the probability measure $\chi^{U}$ on $S^{1}$ given by%
\[
\chi^{U}=\frac{1}{N}(\delta_{\lambda_{1}}+\cdots+\delta_{\lambda_{n}}).
\]

If $U$ is chosen randomly from $U(N)$ with distribution $\rho_{t}^{N},$ then
the empirical eigenvalue measure $\chi^{U}$ is a random measure on $S^{1}.$ In
the large-$N$ limit, however, the empirical eigenvalue distribution ceases to
be random. Rather, $\chi^{U}$ becomes \textit{constant almost surely} with
respect to $\rho_{t}^{N},$ and equal to a certain probability measure $\nu
_{t}$ on $S^{1}.$ The measure $\nu_{t}$ was introduced by Biane in
\cite{Biane1} and various forms of convergence of $\chi^{U}$ to the constant
measure $\nu_{t}$ were established in \cite{Biane1}, \cite{Rains}, and
\cite{Kemp}. (See also work of T. L\'{e}vy \cite{levy} for similar results in
the case of the other families of compact classical groups.)

What this means is that for large $N,$ most of the mass of the heat kernel
measure $\rho_{t}^{N}$ is concentrated on matrices $U$ for which $\chi^{U}$ is
very close (in the weak topology) to the measure $\nu_{t}.$ Thus, most of the
mass of $\rho_{t}^{N}$ is concentrated in a small region in the set of
conjugacy classes, namely the region where the empirical eigenvalue
distributions are close to $\nu_{t}.$

Suppose we consider the transform $B_{t}^{N}$ on class functions, that is,
functions $f:U(N)\rightarrow\mathbb{C}$ that are constant on each conjugacy
class, i.e.,%
\[
f(VUV^{-1})=f(U),
\]
for all $U,V\in U(N).$ The concentration behavior of $\rho_{t}^{N}$ means that
in the large-$N$ limit, all class functions in $L^{2}(U(N),\rho_{t}^{N})$ are
simply constants. (For example, the class function $f(U):=\mathrm{tr}(U^{3})$
becomes equal in the limit to the constant value $\nu_{3}(t),$ where $\nu
_{3}(t)$ is the third moment of Biane's measure $\nu_{t}.$) Thus, at least on
class functions, the scalar transform $B_{t}^{N}$ becomes uninteresting in the limit.

Although one could conceivably get something interesting by considering
complex-valued functions that are not class functions, one could instead
retain simple behavior under conjugation, but extend the transform to
matrix-valued functions. We consider, then, \textbf{conjugation-equivariant
functions}, that is, functions $f:U(N)\rightarrow\mathbb{C}$ satisfying%
\[
f(VUV^{-1})=Vf(U)V^{-1}
\]
for $U,V\in U(N).$ Although the boosted transform $\mathbf{B}_{t}^{N}$ does
not---for any one fixed $N$---preserve the space of single-variable polynomial
functions (see (\ref{btUsquared1})), it does preserve the space of conjugation
equivariant functions.

\begin{proposition}
\label{conjEq.prop}The boosted Segal--Bargmann transform $\mathbf{B}_{t}^{N}$
maps every conjugation-equivariant function on the group $U(N)$ to a
conjugation-equivariant holomorphic function on the group $GL(N;\mathbb{C}).$
\end{proposition}

See Theorem 2.3 in \cite{DHK}. One may now ask what happens to such
conjugation-equivariant functions as the measure $\rho_{t}^{N}$ concentrates
onto a single conjugacy class. This question is answered by the following result.

\begin{proposition}
Suppose $C$ is a conjugacy class in either $U(N)$ or $GL(N;\mathbb{C})$ and
that $f:C\rightarrow M_{N}(\mathbb{C})$ is a conjugation equivariant function.
Then there exists a polynomial $p$ in a single variable such that%
\[
f(A)=p(A)
\]
for all $A\in C.$
\end{proposition}

See Proposition 2.5 in \cite{DHK}. In general, the polynomial $p$ in the
proposition will have degree $N-1.$ In the large-$N$ limit, then, a
conjugation-equivariant function might not be a polynomial, but some sort of
limit of single-variable polynomial functions.

Now, it is not known whether the empirical eigenvalue distribution with
respect to the measures $\mu_{t}^{N}$ on $GL(N;\mathbb{C})$ becomes
deterministic in the large-$N$ limit. (But see related results in
\cite{Kemp}.) Nevertheless, it is shown in Section 4.1 of \cite{DHK} that
traces in $GL(N;\mathbb{C})$ become constant in the limit. Thus, it seems
reasonable to expect that the measures $\mu_{t}^{N}$ also concentrate onto a
single conjugacy class for large $N.$

We have, then, a simple conceptual explanation for Theorem \ref{mainIntro.thm}%
, which asserts that in the large-$N$ limit, $\mathbf{B}_{t}^{N}$ maps
single-variable polynomial functions on $U(N)$ to functions of the same sort
on $GL(N;\mathbb{C}).$ If $p$ is a polynomial and $p_{N}$ is the function on
$U(N)$ obtained by plugging a variable $U\in U(N)$ into $p,$ then $p_{N}$ is
certainly conjugation equivariant. Thus, $\mathbf{B}_{t}^{N}(p_{N})$ is a
conjugation-equivariant function on $GL(N;\mathbb{C}).$ But in the large-$N$
limit, we expect---based on the concentration behavior of the heat kernel
measure $\mu_{t}^{N}$---that \textit{every} conjugation-equivariant function
in $\mathcal{H}L^{2}(GL(N;\mathbb{C}),\mu_{t}^{N};M_{N}(\mathbb{C}))$ is at
least a limit of single-variable polynomial functions.

In our proof in \cite{DHK} of Theorem \ref{mainIntro.thm}, we use the
concentration properties of the heat kernel measures in a more concrete way.
We show that, as will be explained in the remainder of this paper, that the
transform of a single-variable polynomial function on $U(N)$ is a
\textit{trace polynomial} on $GL(N;\mathbb{C}),$ that is, a linear combination
of functions of the form in (\ref{tracePoly}). As $N$ tends to infinity, the
heat kernel measure $\mu_{t}^{N}$ concentrates onto the set where
$\mathrm{tr}(Z^{l})=1$ for all $l.$ (See Theorem \ref{concentrate.thm} for
precise statement of this claim.) Thus, in the large-$N$ limit, trace
polynomials are indistinguishable from single-variable polynomial functions.

\section{The action of the Laplacian on trace polynomials}

We will be interested in the action of $\Delta_{N}$ on \textbf{trace
polynomials}, that is, on matrix-valued functions that are linear combinations
of functions of the form%
\begin{equation}
U^{k}\mathrm{tr}(U)\mathrm{tr}(U^{2})\cdots\mathrm{tr}(U^{n})
\label{tracePoly1}%
\end{equation}
for some $k$ and $n.$ (Actually, we should really consider a more generally
trace \textit{Laurent }polynomials, where we allow negative powers of $U$ and
traces thereof. Nevertheless, for simplicity, I will consider in this paper
only positive powers, which are all that are strictly necessary for the main
results of \cite{DHK}. ) The formula the action of $\Delta_{N}$ on such
functions was originally worked out by Sengupta; see Definition 4.2 and Lemma
4.3 in \cite{Sengupta}. We begin by recording the formula for the Laplacian of
a single power of $U.$

\begin{proposition}
\label{lapPower.prop}For each positive integer $k,$ we have%
\begin{equation}
\Delta_{N}(U^{k})=-kU^{k}-2\sum_{m=1}^{k-1}mU^{m}\mathrm{tr}(U^{k-m}),
\label{lapPower1}%
\end{equation}
and%
\begin{equation}
\Delta_{N}(\mathrm{tr}(U^{k}))=-k\mathrm{tr}(U^{k})-2\sum_{m=1}^{k-1}%
m\mathrm{tr}(U^{m})\mathrm{tr}(U^{k-m}). \label{lapPower2}%
\end{equation}

\end{proposition}

See Theorem 3.3 in \cite{DHK}. Note that when $k=1,$ the sums on the
right-hand sides of (\ref{lapPower1}) and (\ref{lapPower2}) are empty. Thus,
actually, $\Delta_{N}(U)=-U$ and $\Delta_{N}(\mathrm{tr}(U))=-\mathrm{tr}(U).$
Since, by definition, $\Delta_{N}$ acts \textquotedblleft
entrywise\textquotedblright\ on matrix-valued functions, the assertion that
$\Delta_{N}(U)=-U$ is equivalent to the assertion that $\Delta_{N}%
(U_{jk})=-U_{jk}$ for all $j$ and $k.$ A sketch of the proof of this result is
given in Section \ref{magic.sec}.

Let us make a few observations about the formulas in Proposition
\ref{lapPower.prop}. First, since we are supposed to be considering
matrix-valued functions, we should really think of $\mathrm{tr}(U^{k})$ as the
matrix-valued function $U\mapsto\mathrm{tr}(U^{k})I.$ Nevertheless, if we
chose to think of $\mathrm{tr}(U^{k})$ as a scalar-valued function, the
formula in (\ref{lapPower2}) would continue to hold. Second, the Laplacian
$\Delta_{N}$ commutes with applying the trace, so the right-hand side of
(\ref{lapPower2}) is what one obtains by applying the normalized trace to the
right-hand side of (\ref{lapPower1}). Third, the formulas for $\Delta
_{N}(U^{k})$ and $\Delta_{N}(\mathrm{tr}(U^{k}))$ are \textquotedblleft
independent of $N,$\textquotedblright\ meaning that the coefficients of the
various terms on the right-hand side of (\ref{lapPower1}) and (\ref{lapPower2}%
) do not depend on $N.$ This independence holds only because we have chosen to
express things in terms of the \textit{normalized} trace; if we used the
ordinary trace, there would be a factor of $1/N$ in the second term on the
right-hand side of both equations.

Suppose, now, that we wish to apply $\Delta_{N}$ to a product, such as the
function $f(U)=U^{k}\mathrm{tr}(U^{l}).$ As usual with the Laplacian, there is
a product rule that involves three terms, two \textquotedblleft Laplacian
terms\textquotedblright---namely $\Delta_{N}(U^{k})\mathrm{tr}(U^{l})$ and
$U^{k}\Delta_{N}(\mathrm{tr}(U^{l}))$---along with a cross term. The Laplacian
terms can, of course, be computed using (\ref{lapPower1}) and (\ref{lapPower2}%
). The cross term, meanwhile, turns out to be
\[
-\frac{2kl}{N^{2}}U^{k+l}.
\]
Thus, we have%
\[
\Delta_{N}(U^{k}\mathrm{tr}(U^{l}))=\Delta(U^{k})\mathrm{tr}(U^{l}%
)+U^{k}\Delta(\mathrm{tr}(U^{l}))-\frac{2kl}{N^{2}}U^{k+l}.
\]
Again, a sketch of the proof of this result is given in Section
\ref{magic.sec}.

The behavior in the preceding example turns out to be typical: \textit{The
cross term is always of order} $1/N^{2}.$ Thus, to leading order in $N,$ we
may compute the Laplacian of a function of the form (\ref{tracePoly1}) as the
sum of $n+1$ terms, where each term applies the Laplacian to one of the
factors (using (\ref{lapPower1}) or (\ref{lapPower1})) and leaves the other
factors unchanged.

It should be emphasized that this leading-order behavior applies only if (as
in (\ref{tracePoly1})) we have collected together all of the untraced powers
of $U.$ Thus, for example, if we chose to write $U^{5}$ as $U^{3}U^{2},$ it
would \textit{not} be correct to say that $\Delta_{N}(U^{5})$ is $\Delta
_{N}(U^{3})U^{2}+U^{3}\Delta(U^{2})$ plus a term of order $1/N^{2}.$

The smallness of the cross terms leads to the following \textquotedblleft
asymptotic product rule\textquotedblright\ for the action of $\Delta_{N}$ on
trace polynomials.

\begin{proposition}
[Asymptotic product rule]\label{product.prop}Suppose that $f$ and $g$ are
trace polynomials and that either $f$ or $g$ is \textquotedblleft
scalar,\textquotedblright\ meaning that it contains no untraced powers of $U.$
Then%
\[
\Delta_{N}(fg)=\Delta_{N}(f)g+f\Delta_{N}(g)+O(1/N^{2}),
\]
where $O(1/N^{2})$ denotes a fixed trace polynomial multiplied by $1/N^{2}.$
\end{proposition}

The meaning of the expression \textquotedblleft fixed trace
polynomial\textquotedblright\ will be made more precise in the next section.
The assumption that one of the trace polynomials be scalar is essential; if
$f(U)=U^{3}$ and $g(U)=U^{2},$ then the asymptotic product rule does not apply.

The asymptotic product rule may be interpreted as saying that in the situation
of Proposition \ref{product.prop}, the Laplacian \textit{behaves like a
first-order differential operator}. Furthermore, if, say, $f$ is scalar, then
it turns out that $\Delta_{N}^{n}(f)$ is scalar for all $n,$ which means that
we can apply the asymptotic product rule repeatedly. Thus, by a standard power
series argument, together with some simple estimates (Section 4 of
\cite{DHK}), we conclude that%
\begin{equation}
e^{t\Delta_{N}/2}(fg)=e^{t\Delta_{N}/2}(f)e^{t\Delta_{N}/2}(g)+O(1/N^{2}).
\label{expProduct}%
\end{equation}
The asymptotic product rule, along with its exponentiated form
(\ref{expProduct}), is the key to many of the results in \cite{DHK}.

If we restrict our attention to scalar trace polynomials, then the asymptotic
product rule in Proposition \ref{product.prop} will always apply. It is thus
natural to expect that the large-$N$ limit of the action of $\Delta_{N}$ on
scalar trace polynomials can be described by a first-order differential
operator. This expectation is fulfilled in the next section; see Proposition
\ref{diffOps.prop}.

Using the asymptotic product rule, along with Proposition \ref{lapPower.prop},
we can readily compute---to leading order in $N$---the Laplacian of any trace polynomial.

\begin{proposition}
\label{lap1.prop}For any non-negative integers $k$ and $l_{1},\ldots,l_{M},$
we have%
\begin{align*}
\Delta_{N}(U^{k}\mathrm{tr}(U^{l_{1}})\cdots\mathrm{tr}(U^{l_{M}}))  &
=\Delta_{N}(U^{k})\mathrm{tr}(U^{l_{1}})\cdots\mathrm{tr}(U^{l_{M}})\\
&  +U^{k}\Delta_{N}(\mathrm{tr}(U^{l_{1}}))\mathrm{tr}(U^{l_{2}}%
)\cdots\mathrm{tr}(U^{l_{M}})\\
&  +\cdots\\
&  +U^{k}\mathrm{tr}(U^{l_{1}})\cdots\mathrm{tr}(U^{l_{M-1}})\Delta
_{N}(\mathrm{tr}(U^{l_{M}}))\\
&  +O(1/N^{2}),
\end{align*}
where $O(1/N^{2})$ denotes a fixed trace polynomial multiplied by $1/N^{2}.$
\end{proposition}

\section{Polynomials and trace polynomials\label{poly.sec}}

We now give a more precise meaning to the phrase \textquotedblleft fixed trace
polynomial\textquotedblright\ in Propositions \ref{product.prop} and
\ref{lap1.prop}, and thus to the notion of $O(1/N^{2})$ occurring in those
propositions. Along the way, we will explore a subtle distinction between
polynomials and trace polynomials.

\begin{definition}
\label{scalar.def}Let $\mathbb{C}[u,\mathbf{v}]$ denote the space of
polynomials in $u$ and $\mathbf{v},$ where $u$ is a single indeterminate and
where $\mathbf{v}=(v_{1},v_{2},v_{3},\ldots)$ denotes an infinite list of
indeterminates. An element $p$ of $\mathbb{C}[u,\mathbf{v}]$ is said to be
\textbf{scalar} if $p(u,\mathbf{v})$ is independent of $u.$
\end{definition}

Note that by definition of the term \textquotedblleft
polynomial,\textquotedblright\ any given element of $\mathbb{C}[u,\mathbf{v}]
$ depends on only finitely many of the variables $v_{1},v_{2},\ldots.$

\begin{definition}
\label{eval.def}Suppose $p$ is an element of $\mathbb{C}[u,\mathbf{v}].$ Then
for each $N\geq1,$ define the function $p_{N}:U(N)\rightarrow M_{N}%
(\mathbb{C})$ by%
\[
p_{N}(U)=p(U,\mathrm{tr}(U),\mathrm{tr}(U^{2}),\mathrm{tr}(U^{3}),\ldots).
\]
That is, $p_{N}$ is obtained by making the substitution $u=U$ and
$v_{j}=\mathrm{tr}(U^{j}),$ $j=1,2,\ldots.$ Functions of the form $p_{N}$ on
$U(N)$ are called \textbf{trace polynomial functions}, or simply trace polynomials.

A function $f:U(N)\rightarrow M_{N}(\mathbb{C})$ is a \textbf{scalar trace
polynomial} if it can be represented as $f=p_{N}$ where $p\in\mathbb{C}%
[u,\mathbf{v}]$ is independent of $u.$
\end{definition}

In \cite{DHK}, we consider a more general class, in which we allow both
negative powers of $U$ and traces of negative powers of $U.$ For simplicity,
we limit ourselves here to non-negative powers.

It is important to distinguish between the \textquotedblleft
abstract\textquotedblright\ polynomial $p$, which is an element of
$\mathbb{C}[u,\mathbf{v}],$ and the associated trace polynomial
\textit{function} $p_{N}:U(N)\rightarrow M_{N}(\mathbb{C}).$ As it turns out,
it is possible to have a nonzero polynomial $p$ for which $p_{N}=0$ for
certain values of $N.$ In the $N=2$ case, for example, the Cayley--Hamilton
theorem tells us that for all $A\in M_{2}(\mathbb{C}),$ we have%
\[
A^{2}-\mathrm{Trace}(A)A+\det(A)I=0.
\]
Meanwhile, in $M_{2}(\mathbb{C}),$ we have the easily verified identity%
\[
\det(A)=\frac{1}{2}((\mathrm{Trace}(A))^{2}-\mathrm{Trace}(A^{2})).
\]
Thus, restricting to $U(2)$ and writing things in terms of the normalized
trace $\mathrm{tr}(\cdot),$ we have that%
\[
U^{2}-2\mathrm{tr}(U)U+2(\mathrm{tr}(U))^{2}I-\mathrm{tr}(U^{2})I=0
\]
for all $U\in U(2).$

We see, then, that if $p\in\mathbb{C}[u,\mathbf{v}]$ is given by%
\[
p(u,\mathbf{v})=u^{2}-2uv_{1}+2v_{1}^{2}-v_{2},
\]
then the function $p_{2}$ on $U(2)$ is identically zero. There is, however, no
reason that $p_{N}$ should be zero for $N>2.$ Indeed, we show in Section 2.4
of \cite{DHK} that for any $p\in\mathbb{C}[u,\mathbf{v}],$ if $p_{N}$ is
identically zero for \textit{all} $N,$ then $p$ must be the zero polynomial.

Although there is not a one-to-one correspondence between polynomials $p$ and
trace polynomial functions $p_{N},$ it turns out that there is a well-defined
linear operator $\mathcal{D}_{N}$ on $\mathbb{C}[u,\mathbf{v}]$ that
\textquotedblleft intertwines\textquotedblright\ with the action of
$\Delta_{N}$ on functions.

\begin{theorem}
\label{LN.thm}For each $N\geq1,$ there exists a linear operator $\mathcal{D}%
_{N}:\mathbb{C}[u,\mathbf{v}]\rightarrow\mathbb{C}[u,\mathbf{v}]$ such that
for all $p\in\mathbb{C}[u,\mathbf{v}],$ we have%
\[
\Delta_{N}(p_{N})=(\mathcal{D}_{N}p)_{N}.
\]
The operator $\mathcal{D}_{N}$ can be decomposed as%
\begin{equation}
\mathcal{D}_{N}=\mathcal{D}-\frac{1}{N^{2}}\mathcal{L}, \label{DNdecomp}%
\end{equation}
for two linear operators $\mathcal{D}$ and $\mathcal{L}$ mapping
$\mathbb{C}[u,\mathbf{v}]$ to itself.

The operator $\mathcal{D}$ is uniquely determined by the following properties.

\begin{enumerate}
\item \label{uk}$\mathcal{D}(u^{k})=-ku^{k}-2\sum_{m=1}^{k-1}mu^{m}v_{k-m}.$

\item \label{vk}$\mathcal{D}(v_{k})=-kv_{k}-2\sum_{m=1}^{k-1}mv_{m}v_{k-m}.$

\item \label{ppr}For all $p$ and $q$ in $\mathbb{C}[u,\mathbf{v}],$ if either
$p$ or $q$ is scalar, then%
\[
\mathcal{D}(pq)=\mathcal{D}(p)q+p\mathcal{D}(q).
\]

\end{enumerate}
\end{theorem}

This result follows from Theorem 1.18 in \cite{DHK}. Recall that the variable
$u$ is a stand-in for the variable $U$ in a trace polynomial, whereas the
variable $v_{k}$ is a stand-in for $\mathrm{tr}(U^{k}).$ Thus, Points \ref{uk}
and \ref{vk} are simply the polynomial counterparts to Proposition
\ref{lapPower.prop}. Point \ref{ppr}, meanwhile, is simply the polynomial
counterpart to the asymptotic product rule in Proposition \ref{product.prop}.

\begin{proposition}
\label{diffOps.prop}Suppose $p\in\mathbb{C}[u,\mathbf{v}]$ is scalar, that is,
independent of $u.$ Then the action of $\mathcal{D}$ on $p$ is given by%
\[
\mathcal{D}p=-\sum_{k=1}^{\infty}kv_{k}\frac{\partial p}{\partial v_{k}}%
-2\sum_{k=2}^{\infty}\left(  \sum_{j=1}^{k-1}jv_{j}v_{k-j}\right)
\frac{\partial p}{\partial v_{k}}%
\]
and the action of $\mathcal{L}$ on $p$ is given by%
\[
\mathcal{L}p=\sum_{j,k=1}^{\infty}jkv_{k+j}\frac{\partial^{2}p}{\partial
v_{j}\partial v_{k}}.
\]

\end{proposition}

This result follows, again, from Theorem 1.18 in \cite{DHK}. Note that the
actions of $\mathcal{D}$ and $\mathcal{L}$ on scalar polynomials are described
by \textit{differential} operators, and that the leading-order term
$\mathcal{D}$ acts as a \textit{first-order} differential operator. Since the
scalar polynomial $p$ depends on only finitely many of the variables $v_{j},$
only finitely many of the terms in each sum is nonzero. There is also a
formula in that theorem for the action of $\mathcal{D}$ and $\mathcal{L}$ on
nonscalar polynomials (i.e., those polynomials $p(u,\mathbf{v})$ that depend
nontrivially on $u$). The \textquotedblleft full\textquotedblright\ operators
$\mathcal{D}$ and $\mathcal{L}$ are not, however, differential operators. See
Theorem 1.18 in \cite{DHK} for the exact expression.

We may now express the asymptotic product rule more precisely as follows.

\begin{proposition}
[Asymptotic product rule, Version 2]Suppose $p$ and $q$ are polynomials in
$\mathbb{C}[u,\mathbf{v}]$ and that $p $ is scalar (Definition
\ref{scalar.def}). Then there exists a polynomial $r$ such that%
\[
\Delta_{N}(p_{N}q_{N})=\Delta_{N}(p_{N})q_{N}+p_{N}\Delta_{N}(q_{N})+\frac
{1}{N^{2}}r_{N}.
\]

\end{proposition}

\begin{proof}
Apply Theorem \ref{LN.thm} and set $r=-\mathcal{L}(pq).$
\end{proof}

\section{The product rule and concentration of traces}

Recall that a key idea underlying Theorem \ref{mainIntro.thm} is the
phenomenon of \textit{concentration of trace}. Concentration of trace means
that both of the relevant heat kernel measures, $\rho_{t}^{N}$ on $U(N)$ and
$\mu_{t}^{N}$ on $GL(N;\mathbb{C}),$ concentrate in the large-$N$ limit on the
set where the trace of a power is constant. Thus, the function $\mathrm{tr}%
(U^{k})$, as an element of $L^{2}(U(N),\rho_{t}^{N}),$ becomes equal to a
certain constant $\nu_{k}(t)$ in the limit, and similarly for the function
$\mathrm{tr}(Z^{k})$ in $\mathcal{H}L^{2}(GL(N;\mathbb{C}),\mu_{t}^{N}).$ In
this section, we trace the origin of the concentration-of-trace phenomenon to
the asymptotic product rule.

On, say, the $U(N)$ side, the measure $\rho_{t}^{N}$ is the heat kernel
measure at the identity, which means that%
\begin{equation}
\int_{U(N)}f(U)\rho_{t}^{N}(U)~dU=e^{t\Delta/2}(f)(I). \label{intRho}%
\end{equation}
Suppose now that $f$ belongs to some algebra of real-valued functions to which
the asymptotic product rule applies. (For example, $f$ might be the real or
imaginary part of $\mathrm{tr}(U^{k}).$) Then applying the exponentiated form
(\ref{expProduct}) of the product rule with $f=g,$ we obtain%
\begin{equation}
e^{t\Delta_{N}/2}(f^{2})=e^{t\Delta_{N}/2}(f)e^{t\Delta_{N}/2}(f)+O(1/N^{2}).
\label{var1}%
\end{equation}

In light of (\ref{intRho}), (\ref{var1}) reduces to%
\[
\int_{U(N)}f^{2}~d\rho_{t}^{N}=\left(  \int_{U(N)}f~d\rho_{t}^{N}\right)
^{2}+O(1/N^{2}).
\]
In probabilistic language, this says that%
\begin{equation}
E(f^{2})=(E(f))^{2}+O(1/N^{2}), \label{Ef2}%
\end{equation}
where $E$ denotes expectation value with respect to the measure $\rho_{t}%
^{N}.$

Recall that the \textit{variance} of $f$ is defined as
\[
\mathrm{Var}(f):=E((f-E(f))^{2}),
\]
and may be computed as $\mathrm{Var}(f)=E(f^{2})-(E(f))^{2}.$ Thus,
(\ref{Ef2}) is telling us that%
\[
\mathrm{Var}(f)=O(1/N^{2}).
\]
Thus, when $N$ is large, $f(U)$ is close to the constant value $E(f)$ for most
values of $U.$

\begin{conclusion}
\label{concentrate.conclusion}Suppose $f$ belongs to some algebra of
real-valued functions on $U(N)$ for which the asymptotic product rule applies.
Then the variance of $f$ with respect to the heat kernel measure $\rho_{t}%
^{N}$ is small for large $N.$
\end{conclusion}

We may apply Conclusion \ref{concentrate.conclusion} with $f$ being the real
or imaginary part of $\mathrm{tr}(U^{k}).$ We conclude that $\mathrm{tr}%
(U^{k})$---as an element of $L^{2}(U(N),\rho_{t}^{N})$---is concentrating onto
its expectation value for large $N.$ A similar argument shows that
$\mathrm{tr}(Z^{k})$---as an element of $\mathcal{H}L^{2}(GL(N;\mathbb{C}%
),\mu_{t}^{N})$---is concentrating onto the value 1 for large $N.$ In
\cite{DHK}, we prove the following more general result.

\begin{theorem}
[Concentration of Traces]\label{concentrate.thm}For any polynomial
$p\in\mathbb{C}[u,\mathbf{v}],$ let $\pi_{t}:\mathbb{C}[u,\mathbf{v}%
]\rightarrow\mathbb{C}[u]$ be the \textbf{trace evaluation map} obtained by
setting each of the variables $v_{j}$ equal to the constant value $\nu
_{k}(t),$ where $\nu_{k}(t)$ is the $k$th moment of Biane's measure $\nu_{t}$
on $S^{1}.$ That is,%
\[
(\pi_{t}p)(u)=p(u,\nu_{1}(t),\nu_{2}(t),\ldots).
\]
Since $\nu_{k}(0)=1,$ the map $\pi_{0}$ corresponds to evaluating each of the
variables $v_{j}$ to the value 1. Then we have the following results:%
\begin{align*}
\lim_{N\rightarrow\infty}\left\Vert p_{N}-(\pi_{t}p)_{N}\right\Vert
_{L^{2}(U(N),\rho_{t}^{N};M_{N}(\mathbb{C}))}  &  =0\\
\lim_{N\rightarrow\infty}\left\Vert p_{N}-(\pi_{0}p)_{N}\right\Vert
_{L^{2}(GL(N;\mathbb{C}),\mu_{t}^{N};M_{N}(\mathbb{C}))}  &  =0,
\end{align*}
where the notation $p_{N}$ is as in Definition \ref{eval.def}.
\end{theorem}

This is the $s=t$ case of Theorem 1.16 in \cite{DHK}. We have seen in this
section that the phenomenon of concentration of trace can be understood as a
consequence of the asymptotic product rule. In the next section, we will use
the asymptotic product rule to compute---to leading order in $N$---the value
of $e^{t\Delta_{N}/2}(U^{k}).$

\section{A recursive approach to the Segal--Bargmann transform on
polynomials\label{induct.sec}}

The operator $\mathcal{D}$ in Theorem \ref{LN.thm} describes the leading-order
behavior of $\Delta_{N}$ on trace polynomials (see (\ref{DNdecomp})). Thus,
$e^{t\mathcal{D}/2}$ describes the leading-order behavior of the
Segal--Bargmann transform $\mathbf{B}_{t}^{N}$ on trace polynomials. In this
section (following Section 5.1 of \cite{DHK}), we construct a recursive method
of computing $e^{t\mathcal{D}/2}(u^{k})$ for positive integers $k.$ Since
$\mathcal{D}(u)=-u,$ our base case is $e^{t\mathcal{D}/2}(u)=e^{-t/2}u.$ The
induction step will use the product rule for $\mathcal{D}$ (Point \ref{ppr} of
Theorem \ref{LN.thm}) in an essential way.

Given a monomial $q$ in $\mathbb{C}[u,\mathbf{v}],$ say%
\[
q(u,\mathbf{v})=u^{l_{0}}v_{1}^{l_{1}}\cdots v_{M}^{l_{M}},
\]
for some $M,$ we define the \textbf{trace degree} of $q$ to be%
\[
\mathrm{\deg}(q)=l_{0}+l_{1}+2l_{2}+\cdots+Ml_{M}.
\]
This definition reflects the idea that $v_{k}$ is a stand-in for the function
$\mathrm{tr}(U^{k})$ on $U(N).$ Thus, the trace degree of $q$ is the total
number of factors of $U$ in the associated trace polynomial $q_{N}(U).$ (Thus,
for example, $q(u,\mathbf{v}):=u^{2}v_{2}^{2}$ has trace degree 6 because the
associated trace polynomial $q_{N}(U)=U^{2}(\mathrm{tr}(U^{2}))^{2}$ has six
factors of $U.$) We say that a polynomial $p\in\mathbb{C}[u,\mathbf{v}]$ is
homogeneous of trace degree $k$ if $p$ is a linear combination of monomials
having trace degree $k.$

Let $\mathbb{C}^{(k)}[u,\mathbf{v]}$ denote the space of $p\in\mathbb{C}%
[u,\mathbf{v}]$ that are homogeneous of trace degree $k,$ so that
$\mathbb{C}[u,\mathbf{v}]$ is the direct sum of the $\mathbb{C}^{(k)}%
[u,\mathbf{v}]$'s, for $k=0,1,2,\ldots.$ Each space $\mathbb{C}^{(k)}%
[u,\mathbf{v]}$ is easily seen to be finite dimensional, and is invariant
under the operators $\mathcal{D}$ and $\mathcal{L}$ in Theorem \ref{LN.thm}.
Thus, it makes sense to exponentiate any linear combination of these operators
by thinking of them as operators on each of the finite-dimensional spaces
$\mathbb{C}^{(k)}[u,\mathbf{v}].$

Let $N$ be the \textquotedblleft number operator\textquotedblright\ on
$\mathbb{C}[u,\mathbf{v}],$ namely, the operator such that%
\[
\left.  N\right\vert _{\mathbb{C}^{(k)}[u,\mathbf{v}]}=kI.
\]
It is convenient to decompose $\mathcal{D}$ as%
\begin{equation}
\mathcal{D}=-N+\mathcal{\tilde{D}}. \label{LinfDecomp}%
\end{equation}
Since the polynomials $p(u,\mathbf{v}):=u^{k}$ and $q(u,\mathbf{v}):=v_{k}$
both belong to $\mathbb{C}^{(k)}[u,\mathbf{v}],$ if we with to compute
$\mathcal{\tilde{D}}(u^{k})$ or $\mathcal{\tilde{D}}(v_{k}),$ we simply omit
the term of $-ku^{k}$ or $-kv_{k}$ in front of the sums in Points \ref{uk} and
\ref{vk} of Theorem \ref{LN.thm}.

The two terms on the right-hand side of (\ref{LinfDecomp}) commute, since they
commute on $\mathbb{C}^{(k)}[u,\mathbf{v}]$ for each $k.$ Thus,%
\[
e^{t\mathcal{D}/2}=e^{t\mathcal{\tilde{D}}/2}e^{-tN/2}.
\]
In particular,%
\begin{equation}
e^{t\mathcal{D}/2}(u^{k})=e^{-tk/2}e^{t\mathcal{\tilde{D}}/2}(u^{k}).
\label{expLD}%
\end{equation}
Now,%
\begin{align*}
\frac{d}{dt}e^{t\mathcal{\tilde{D}}/2}(u^{k})  &  =\frac{1}{2}%
e^{t\mathcal{\tilde{D}}/2}(\mathcal{\tilde{D}}u^{k})\\
&  =\frac{1}{2}e^{t\mathcal{\tilde{D}}/2}\left(  -2\sum_{m=1}^{k-1}%
mu^{m}v_{k-m}\right)  ,
\end{align*}
by Point \ref{uk} of Theorem \ref{LN.thm}.

Since the polynomial $q(u,\mathbf{v})=v_{k-m}$ is scalar, the product rule
applies to the product $u^{m}v_{k-m}$. Since, also, $\mathcal{\tilde{D}}%
^{n}(v_{k-m})$ is scalar for all $n,$ we may apply a standard power series
argument to show that $e^{t\mathcal{\tilde{D}}/2}$ behaves multiplicatively on
the product $u^{m}v_{k-m}.$ Thus,
\begin{equation}
\frac{d}{dt}e^{t\mathcal{\tilde{D}}/2}(u^{k})=-\sum_{m=1}^{k-1}%
me^{t\mathcal{\tilde{D}}/2}(u^{m})e^{t\mathcal{\tilde{D}}/2}(v_{k-m}).
\label{diff1}%
\end{equation}
A similar argument shows that
\begin{equation}
\frac{d}{dt}e^{t\mathcal{\tilde{D}}/2}(v_{k})=-\sum_{m=1}^{k-1}%
me^{t\mathcal{\tilde{D}}/2}(v_{m})e^{t\mathcal{\tilde{D}}/2}(v_{k-m}).
\label{diff2}%
\end{equation}

We may then integrate either of equations (\ref{diff1}) or (\ref{diff2}), with
initial condition determined by the fact that $e^{t\mathcal{\tilde{D}}/2}=I$
when $t=0.$ This gives the following result.

\begin{theorem}
\label{induct.thm}For all positive integers $k,$ we have the recursive
formulas%
\begin{align}
e^{t\mathcal{\tilde{D}}/2}(u^{k}) &  =u^{k}-\sum_{m=1}^{k-1}m\int_{0}%
^{t}e^{s\mathcal{\tilde{D}}/2}(u^{m})e^{s\mathcal{\tilde{D}}/2}(v_{k-m}%
)~ds\nonumber\\
e^{t\mathcal{\tilde{D}}/2}(v_{k}) &  =v_{k}-\sum_{m=1}^{k-1}m\int_{0}%
^{t}e^{s\mathcal{\tilde{D}}/2}(v_{m})e^{s\mathcal{\tilde{D}}/2}(v_{k-m}%
)~ds.\label{induct}%
\end{align}

\end{theorem}

Since in the sums, both $m$ and $k-m$ are always strictly smaller than $k,$ we
can assume, recursively, that $e^{s\mathcal{\tilde{D}}/2}(u^{m}),$
$e^{s\mathcal{\tilde{D}}/2}(v_{k-m}),$ and $e^{s\mathcal{\tilde{D}}/2}(v_{m})$
are all \textquotedblleft known.\textquotedblright\ Let us now use the
recursion to compute a simple example. It follows from Points (\ref{uk}) and
(\ref{vk}) of Theorem \ref{LN.thm} that $\mathcal{\tilde{D}}%
(u)=\mathcal{\tilde{D}}(v_{1})=0,$ so that $e^{t\mathcal{\tilde{D}}/2}(u)=u$
and $e^{t\mathcal{\tilde{D}}/2}(v_{1})=v_{1}.$ Applying (\ref{induct}) with
$k=2$ then gives
\begin{align*}
e^{t\mathcal{\tilde{D}}/2}(u^{2}) &  =u^{2}-\int_{0}^{t}e^{s\mathcal{\tilde
{D}}/2}(u)e^{s\mathcal{\tilde{D}}/2}(v_{1})~ds\\
&  =u^{2}-\int_{0}^{t}uv_{1}~ds\\
&  =u^{2}-tuv_{1}.
\end{align*}
By (\ref{expLD}), we then have%
\begin{equation}
e^{t\mathcal{D}/2}(u^{2})=e^{-t}(u^{2}-tuv_{1}).\label{expLu2}%
\end{equation}
Similarly, we obtain%
\begin{equation}
e^{t\mathcal{D}/2}(v_{2})=e^{-t}(v_{2}-tv_{1}^{2}).\label{expLv2}%
\end{equation}
The results in (\ref{expLu2}) and (\ref{expLv2}) can then be fed into the
induction procedure in (\ref{induct}) to compute $e^{t\mathcal{D}/2}(u^{3})$
and $e^{t\mathcal{D}/2}(v_{3})$, and so on.

Recalling that $\mathcal{D}$ describes the leading-order behavior of
$\Delta_{N}$ on polynomials, (\ref{expLu2}) tells us that%
\[
\mathbf{B}_{t}^{N}(f)(Z)\approx e^{-t}(Z^{2}-tZ\mathrm{tr}(Z)),
\]
where $\approx$ indicates that the norm (in $L^{2}(GL(N;\mathbb{C}),\mu
_{t}^{N};M_{N}(\mathbb{C}))$) of the difference is small. Since, also, a
concentration-of-trace phenomenon tells us that $\mathrm{tr}(Z)\approx1$
(Theorem \ref{concentrate.thm}), we have%
\[
\mathbf{B}_{t}^{N}(f)(Z)\approx e^{-t}(Z^{2}-tZ).
\]
Thus, if $p(u)=u^{2},$ the polynomial $q_{t}$ in Theorem \ref{mainIntro.thm}
is%
\[
q_{t}(z)=e^{-t}(z^{2}-tz),
\]
as claimed in the overview.

More generally, suppose that $p(u)=u^{k}.$ We may compute the associated
polynomial $q_{t}$ by the following two-step process. First, we compute,
inductively, $e^{t\mathcal{\tilde{D}}/2}(u^{k})$---and thus, by \ref{expLD},
$e^{t\mathcal{D}/2}$---using the recursion in (\ref{induct}). Second, we
evaluate each of the variables $v_{k}$ in the expression for $e^{t\mathcal{D}%
/2}(u^{k})$ to the value 1. (Recall that $v_{k}$ is a stand-in for
$\mathrm{tr}(Z^{k})$ and that $\mathrm{tr}(Z^{k})\approx1$ in $L^{2}%
(GL(N;\mathbb{C}),\mu_{t}^{N}).$) We have carried out these computations in
Mathematica with the result that if
\[
p(u)=u^{4}
\]
the polynomial $q_{t}$ in Theorem \ref{mainIntro.thm} is given by%
\[
q_{t}(z)=e^{-2t}\left[  z^{4}-tz^{3}+(4t^{2}-2t)z^{2}+\left(  -\frac{8}%
{3}t^{3}+4t^{2}-t\right)  z\right]  .
\]

The recursive procedure in Theorem \ref{induct.thm} allows us to compute the
heat operator applied to any positive power of $U.$ Using this result, we can
also compute the heat operator applied to a negative power of $U.$ It is
easily seen that the heat operator (as applied to functions $f:U(N)\rightarrow
M_{N}(\mathbb{C})$) commutes with taking adjoints:%
\[
e^{t\Delta_{N}/2}(f^{\ast})=(e^{t\Delta_{N}/2}f)^{\ast}.
\]
Since $U^{-k}=(U^{k})^{\ast}$ for $U\in U(N),$ we see that
\[
e^{t\Delta_{N}/2}(U^{-k})=(e^{t\Delta_{N}/2}U^{k})^{\ast}.
\]

Using this line of reasoning, we can easily prove an analog of Theorem
\ref{mainIntro.thm} for negative powers of $U.$ If $p$ is a polynomial in a
single variable, we can define $p^{N}:U(N)\rightarrow M_{N}(\mathbb{C})$ by
substituting $U^{-1},$ rather than $U,$ into $p.$ Then if $q_{t}$ is the same
polynomial as in Theorem \ref{mainIntro.thm}, the theorem holds with $p_{N}$
replaced by $p^{N}$:%
\[
\lim_{N\rightarrow\infty}\left\Vert \mathbf{B}_{t}^{N}(p^{N})-(q_{t}%
)^{N}\right\Vert _{L^{2}(GL(N;\mathbb{C}),\mu_{t}^{N};M_{N}(\mathbb{C}))}=0.
\]

\section{The magic formulas for computing the Laplacian\label{magic.sec}}

In this section, we explain how one can evaluate $\Delta_{N}$ on trace
polynomial functions. In particular, we will see the origin of the asymptotic
product rule.

\begin{theorem}
\label{magic.thm}Let $\{X_{j}\}$ be any orthonormal basis for $u(N)$ with
respect to the inner product in (\ref{unInner}). Then for all $A,B\in
M_{N}(\mathbb{C})$ we have%
\begin{align}
\sum_{j}X_{j}^{2}  &  =-I\label{magic1}\\
\sum_{j}X_{j}AX_{j}  &  =-\mathrm{tr}(A)I\label{magic2}\\
\sum_{j}\mathrm{tr}(X_{j}A)X_{j}  &  =-\frac{1}{N^{2}}A\label{magic3}\\
\sum_{j}\mathrm{tr}(X_{j}A)\mathrm{tr}(X_{j}B)  &  =-\frac{1}{N^{2}%
}\mathrm{tr}(AB), \label{magic4}%
\end{align}
where $\mathrm{tr}(\cdot)$ is the normalized trace in (\ref{normalizedTrace}).
\end{theorem}

These \textquotedblleft magic formulas\textquotedblright\ are established in
Lemma 4.1 of \cite{Sengupta}. One can prove the formulas by first establishing
that the sums are independent of the choice of orthonormal basis and then
computing by brute force in one particular basis. (See also Section 3.1 of
\cite{DHK}.) Note the presence of a factor of $1/N^{2}$ on the right-hand
sides of (\ref{magic3}) and (\ref{magic4}).

\begin{proposition}
\label{lapPoly.prop}For any non-negative integer $k$ and any (possibly empty)
sequence $l_{1},\ldots,l_{M}$ of positive integers, we have%
\[
\Delta_{N}(U^{k}\mathrm{tr}(U^{l_{1}})\cdots\mathrm{tr}(U^{l_{M}}%
))=\mathrm{I}+\mathrm{II}+\mathrm{III},
\]
where%
\begin{align*}
\mathrm{I}  &  =\Delta_{N}(U^{k})\mathrm{tr}(U^{l_{1}})\cdots\mathrm{tr}%
(U^{l_{M}})\\
&  +\sum_{j=1}^{M}\mathrm{tr}(U^{l_{1}})\cdots\widehat{\mathrm{tr}(U^{l_{j}}%
)}\cdots\mathrm{tr}(U^{l_{M}})\cdot\Delta_{N}(\mathrm{tr}(U^{l_{j}})),
\end{align*}
and%
\[
\mathrm{II}=-\frac{2}{N^{2}}U^{k}\sum_{j<m}l_{j}l_{m}\mathrm{tr}(U^{l_{1}%
})\cdots\widehat{\mathrm{tr}(U^{l_{j}})}\cdots\widehat{\mathrm{tr}(U^{l_{m}}%
)}\cdots\mathrm{tr}(U^{l_{M}})\cdot\mathrm{tr}(U^{l_{j}+l_{k}})
\]
and%
\[
\mathrm{III}=-\frac{2}{N^{2}}\sum_{j=1}^{M}kl_{j}U^{k+l_{j}}\mathrm{tr}%
(U^{l_{1}})\cdots\widehat{\mathrm{tr}(U^{l_{j}})}\cdots\mathrm{tr}(U^{l_{M}%
}).
\]

\end{proposition}

Proposition \ref{lapPoly.prop} is a slight strengthening of Lemma 4.3 in
\cite{Sengupta}. This result can be combined with Proposition
\ref{lapPower.prop} to obtain an explicit formula for the Laplacian of any
trace polynomial. In particular, from Propositions \ref{lapPower.prop} and
\ref{lapPoly.prop}, we can easily obtain the operators $\mathcal{D}$ and
$\mathcal{L}$ in Theorem \ref{LN.thm}.

Note that we \textit{do not} assume the exponents $l_{1},\ldots,l_{M}$ are
distinct. Term I in the proposition is the term in which $\Delta_{N}$ behaves
like a first-order operator; that is, in Term I, we apply the Laplacian to
each factor separately. Since the remaining terms have a factor of $1/N^{2}$
in front, the asymptotic product rule follows from the proposition. In each
entry in Term II, we combine two separate traces, $\mathrm{tr}(U^{l_{j}})$ and
$\mathrm{tr}(U^{l_{k}}),$ into a single trace, $\mathrm{tr}(U^{l_{j}+l_{k}}).$
In each entry in Term III, we combine $U^{k}$ and $\mathrm{tr}(U^{l_{j}})$
into the factor of $U^{k+l_{j}}.$ Terms II and III in Proposition
\ref{lapPoly.prop} arise from (\ref{magic3}) and (\ref{magic4}) in Theorem
\ref{magic.thm}.

We now illustrate the proofs of Propositions \ref{lapPower.prop} and
\ref{lapPoly.prop} by verifying one example of each proposition, using the
magic formulas. It requires only a bit of combinatorics to prove the general
results by the same method.

\textbf{First example}. We illustrate the proof of Proposition
\ref{lapPower.prop} by considering the function
\[
f(U)=U^{2}.
\]
Given a basis element $X_{j}$ in $u(N),$ we may compute the associated
left-invariant vector field $\tilde{X}_{j}$, as in (\ref{leftInv}), using the
product rule:%
\begin{align*}
(\tilde{X}_{j}f)(U)  &  =\left.  \frac{d}{ds}Ue^{sX_{j}}Ue^{sX_{j}}\right\vert
_{s=0}\\
&  =UX_{j}U+U^{2}X_{j}.
\end{align*}
Applying $\tilde{X}_{j}$ again gives%
\[
(\tilde{X}_{j}^{2}f)(U)=UX_{j}^{2}U+UX_{j}UX_{j}+UX_{j}UX_{j}+U^{2}X_{j}^{2}.
\]
To compute $\Delta_{N}f,$ we sum over $j$ and use the magic formulas
(\ref{magic1}) and (\ref{magic2}), with the result that%
\[
(\Delta_{N}f)(U)=-2U^{2}-2U\mathrm{tr}(U).
\]

\textbf{Second example. }We illustrate the proof of Proposition
\ref{lapPoly.prop} by considering the function%
\[
f(U)=U^{2}\mathrm{tr}(U^{2}).
\]
We apply $\tilde{X}_{j}$ as in the previous example, giving%
\begin{align*}
(\tilde{X}_{j}f)(U)  &  =UX_{j}U\mathrm{tr}(U^{2})+U^{2}X_{j}\mathrm{tr}%
(U^{2})\\
&  +U^{2}\mathrm{tr}(UX_{j}U)+U^{2}\mathrm{tr}(U^{2}X_{j}).
\end{align*}
Applying $\tilde{X}_{j}$ a second time gives a total of ten terms:%
\begin{align}
(\tilde{X}_{j}^{2}f)(U)  &  =UX_{j}^{2}U\mathrm{tr}(U^{2})+2UX_{j}%
UX_{j}\mathrm{tr}(U^{2})+U^{2}X_{j}^{2}\mathrm{tr}(U^{2})\nonumber\\
&  +U^{2}\mathrm{tr}(UX_{j}^{2}U)+2U^{2}\mathrm{tr}(UX_{j}UX_{j}%
)+U^{2}\mathrm{tr}(U^{2}X_{j}^{2})\nonumber\\
&  +2UX_{j}U\mathrm{tr}(UX_{j}U)+2UX_{j}U\mathrm{tr}(U^{2}X_{j})\nonumber\\
&  +2U^{2}X_{j}\mathrm{tr}(UX_{j}U)+2U^{2}X_{j}\mathrm{tr}(U^{2}X_{j}).
\label{Xj2}%
\end{align}

We now sum (\ref{Xj2}) over $j$ to obtain $\Delta_{N}f.$ In the first line on
the right-hand side of (\ref{Xj2}), all the derivatives are on the $U^{2}$
factor in front of the trace. Thus, after summing over $j,$ the first line
gives $\Delta_{N}(U^{2})\mathrm{tr}(U^{2}).$ Similarly, the second line on the
right-hand side of (\ref{Xj2}) sums to $U^{2}\Delta_{N}(\mathrm{tr}(U^{2})).$
In the remaining two lines, we move the scalar trace factor next to the factor
of $X_{j}$ outside the trace. Then we cyclically permute the matrices inside
the trace to put the factor of $X_{j}$ first. At that point, we can apply the
magic formula (\ref{magic3}), with the result that each of the four terms on
the left-hand side of (\ref{Xj2}) sums to $-(2/N^{2})U^{4}. $ Thus,
\[
(\Delta_{N}f)(U)=\Delta_{N}(U^{2})\mathrm{tr}(U^{2})+U^{2}\Delta
_{N}(\mathrm{tr}(U^{2}))-\frac{8}{N^{2}}U^{4}.
\]
This result agrees with Proposition \ref{lapPoly.prop}, with Term II being
zero in this case.\bigskip

In general, we can understand the asymptotic product rule this way. Suppose we
want to apply $\Delta_{N}$ to a trace monomial $U^{k}\mathrm{tr}(U^{l_{1}%
})\cdots\mathrm{tr}(U^{l_{M}}).$ If apply the vector field $\tilde{X}_{j}$
twice, we get a large number of terms, each of which has two factors of
$X_{j}$ inserted among the various powers of $U$ in the original function. We
first consider all the terms in which both factors of $X_{j}$ reside in the
same power of $U,$ either both inside the factor of $U^{k}$ or both within the
same trace. After summing on $j,$ these terms will simply apply $\Delta_{N}$
to each one of the factors, either to $U^{k}$ or to $\mathrm{tr}(U^{l_{n}})$
for some $n.$ In the remaining terms, we have the two factors of $X_{j}$ in
two different powers of $U,$ either one in the $U^{k}$ factor and one in one
of the trace factors, or in two different trace factors. In these cases, we
apply the magic formulas (\ref{magic3}) and (\ref{magic4}), both of which have
a factor of $1/N^{2}$ on the right-hand side. Thus, all deviations from
(first-order) product rule behavior are of order $1/N^{2}.$

\section{The two-parameter transform and the generating
function\label{twoParam.sec}}

In \cite{DHK}, we actually consider the \textit{two-parameter} version of the
Segal--Bargmann transform for $U(N),$ as introduced in \cite{DriverHall} and
\cite{newform}. This transform is the unitary map%
\[
B_{s,t}^{N}:L^{2}(U(N),\rho_{s}^{N})\rightarrow\mathcal{H}L^{2}%
(GL(N;\mathbb{C}),\mu_{s,t}^{N})
\]
given by
\[
B_{s,t}^{N}(f)=(e^{t\Delta_{N}/2}f)_{\mathbb{C}},
\]
where $\mu_{s,t}^{N}$ is a certain heat kernel measure on $GL(N;\mathbb{C}), $
and where $s$ and $t$ are positive numbers with $s>t/2.$ Note that the formula
for $B_{s,t}^{N}$ is the same as for $B_{t}^{N}$; only the measures used on
the domain and range spaces depend on the second parameter, $s.$ We may boost
the transform $B_{s,t}^{N}$ to a transform $\mathbf{B}_{s,t}^{N}$ acting on
matrix-valued functions, precisely as in the case of $B_{t}^{N}.$ In
\cite{DHK}, we prove a version of Theorem \ref{mainIntro.thm} for the two
parameter transform in which $\mathbf{B}_{t}^{N}$ is replaced by
$\mathbf{B}_{s,t}^{N}$ and the polynomial $q_{t}$ is replaced by a polynomial
$q_{s,t}.$

The introduction of the second parameter in \cite{DHK} is not merely to prove
a more general result. Rather, this parameter is critical to establishing
certain properties of the \textit{one-parameter} polynomial map $p\mapsto
q_{t}$ in Theorem \ref{mainIntro.thm}. Specifically, we prove that this map
coincides with the \textquotedblleft free Hall transform\textquotedblright%
\ $\mathcal{G}^{t}$ of Biane \cite{Biane2}. 

The way we prove this is as follows. In Section 5.3 of \cite{DHK} we consider
the polynomial $p_{k}^{s,t}$ for which the associated polynomial $q_{s,t}$ is
simply $z^{k},$ and we then consider the generating function for this family
of polynomials:%
\begin{equation}
\phi^{s,u}(z,t):=\sum_{k=1}^{\infty}p_{k}^{s,t}(u)z^{k}.\label{phiGen}%
\end{equation}
This generating function \textquotedblleft encodes\textquotedblright%
\ polynomials $p_{k}^{s,t},$ in the sense that those polynomials can be
computed by evaluating the Taylor coefficients of $\phi^{s,u}(z,t)$ in the $z$
variable. It is actually necessary to consider two additional generating
functions,
\begin{equation}
\psi^{s}(t,z):=\sum_{k=1}^{\infty}\mathrm{tr}\left(  p_{k}^{s,t}(u)\right)
z^{k}\label{psiGen}%
\end{equation}
and%
\begin{equation}
\rho(s,z):=\sum_{k=1}^{\infty}\mathrm{tr}(u^{k})z^{k}.\label{rhoGen}%
\end{equation}
In (\ref{psiGen}) and (\ref{rhoGen}), the trace is evaluated using the
large-$N$ limit of the expectation value of $\mathrm{tr}(U^{k})$ with respect
to the measure $\rho_{s}$; these limiting expectation values were computed
explicitly by Biane \cite{Biane1}.

In Proposition 5.10 of \cite{DHK}, we use (the two-parameter version of) the
recursion in Section \ref{induct.sec} to obtain the following set of partial
differential equations for the generating functions in the previous paragraph,
together with the appropriate initial conditions.

\begin{proposition}
The generating functions in (\ref{phiGen}), (\ref{psiGen}), and (\ref{rhoGen})
satisfy the following holomorphic PDE's, for sufficiently small $z$:%
\[
\]%
\begin{align*}
\frac{\partial\rho(s,z)}{\partial s}  & =-s\rho\frac{\partial\rho}{\partial
z},\quad\rho(0,z)=\frac{z}{1-z}\\
\frac{\partial\psi^{s}(t,z)}{\partial t}  & =z\psi^{s}\frac{\partial\psi^{s}%
}{\partial z},\quad\psi^{s}(0,z)=\rho(s,e^{-s/2}z)\\
\frac{\partial\phi^{s,u}(t,z)}{\partial t}  & =z\psi^{s}\frac{\partial
\phi^{s,u}}{\partial z},\quad\phi^{s,u}(0,z)=\frac{uz}{1-uz}.
\end{align*}

\end{proposition}

Note that the second and third equations involve the derivatives of $\phi$ and
$\psi$ \textit{with respect to }$t$\textit{\ with }$s$\textit{\ fixed}, an
approach that would not make sense if we considered only the $s=t$ case. We
solve this system of differential equation by the method of characteristics,
with the result (Theorem 1.17 of \cite{DHK}) that $\phi^{s,u}$ is given by the
following implicit formula:%
\[
\phi^{s,u}\left(  t,ze^{\frac{1}{2}(s-t)\frac{1+z}{1-z}}\right)  =\left(
1-uze^{\frac{s}{2}\frac{1+z}{1-z}}\right)  ^{-1}-1.
\]
In particular, when $s=t,$ we obtain the explicit formula%
\begin{equation}
\phi^{t,u}\left(  t,z\right)  =\left(  1-uze^{\frac{t}{2}\frac{1+z}{1-z}%
}\right)  ^{-1}-1.\label{phittGen}%
\end{equation}
The expression in (\ref{phittGen}) is precisely the generating function for
Biane's transform $\left(  \mathcal{G}^{t}\right)  ^{-1}$ (after correcting a
typographical error in \cite{Biane2}).

\end{document}